\documentclass[12pt]{amsart}

\usepackage{tikz} 
\usepackage{amsmath}
\usepackage{amssymb,amsfonts,mathrsfs}
\usepackage[colorlinks=true,linkcolor=blue,citecolor=blue]{hyperref}
\usepackage[driver=pdftex,margin=3cm,heightrounded=true,centering]{geometry}

\usepackage{euler}

\usepackage{mllocal} 

\numberwithin{equation}{section}
\date{\today}

\begin{document}

\title[Regularized limit of determinants for discrete tori]
{Regularized limit of determinants for discrete tori}

\date{\today}

\author{Boris Vertman} 
\address{Institute of Mathematics and Computer Science, University of M\"unster, } 
\address{Einsteinstrasse 62, 48149 M\"unster, Germany} 
\email{vertman@uni-muenster.de}
\urladdr{http://wwwmath.uni-muenster.de/42/arbeitsgruppen/ag-differentialgeometrie/}

\date{This document was compiled on: \today}

\begin{abstract}
We consider a combinatorial Laplace operator on a sequence of discrete graphs which approximates
the $m$-dimensional torus when the discretization parameter tends to infinity. 
We establish a polyhomogeneous expansion of the resolvent trace for the family 
of discrete graphs, jointly in the resolvent and the discretization parameter. 
Based on a result about interchanging regularized limits and regularized
integrals, we compare the regularized limit of the log-determinants of the combinatorial Laplacian on 
the sequence of discrete graphs with the logarithm 
of the zeta determinant for the Laplace Beltrami operator
on the $m$-dimensional torus. In a similar manner we may apply our method to
compare the product of the first $N\in \N$ non-zero eigenvalues of the Laplacian on a torus (or any other 
smooth manifold with an explicitly known spectrum) with the zeta-regularized 
determinant of the  Laplacian in the regularized limit as $N\to \infty$.

\end{abstract}

\maketitle
\tableofcontents

\section{Introduction and formulation of the results}\label{sec-intro}

Introduction of a zeta-regularized determinant for the Hodge Laplacian on compact
Riemannian manifolds by Ray and Singer \cite{RS} provided a counterpart to 
the determinant of discrete Laplacian on a simplicial complex associated to a triangulation 
of the manifold. Relation between the zeta-regularized determinant of the Hodge Laplacian 
as an analytic spectral invariant, and the determinant of the discrete Laplacian as a combinatorial
quantity, has been shown to go beyond being just formal counterparts by the proof
of the Ray-Singer conjecture \cite{RS} by Cheeger \cite{Che} and M\"uller \cite{Mue}.
\medskip

In fact, M\"uller \cite{Mue} proved that a specific combination of determinants for 
discrete Laplacians in various degrees, defined on simplicial complexes associated to a triangulation 
of a compact Riemannian manifold, converges to the corresponding combination 
of zeta-regularized determinants of Hodge Laplacians when the mesh of the triangulation goes to zero. 
Another instance of a link between determinants of discrete Laplacians
and zeta-regularized determinants of the corresponding Hodge Laplacians is 
a recent joint work with Reshetikhin \cite{ReVe} which in part motivated 
the Burghelea Friedlander Kappeler gluing formula for determinants by studying the 
corresponding combinatorial problem. \medskip

In both instances the behavior of the individual determinants remained open, since the 
discussion is rather based on existence of a well-defined limit for \emph{combinations} of 
determinants for discrete Laplacians under finer discretizations. This leads to the general problem
if the zeta-regularized determinant of a Hodge Laplacian may indeed be recovered from 
its discretization. This translates into a question on existence of an asymptotic expansion for 
determinants of discrete Laplacians under refinement of discretization.
\medskip 
 
Interest in the asymptotic behaviour for determinants of discrete Laplacians 
arises in various mathematical settings, even without the conjectured relationship 
with its zeta-regularized counterpart. In statistical mechanics
the interest stems from identification of the determinant in terms of the number of
spanning trees on a graph by Kirchhoff \cite{Kir}. Moreover, in the setting 
of two-dimensional lattices, determinants of certain $\Z^2$ subgraphs were expressed
in terms of the number of dimer coverings of related $\Z^2$ subgraphs by Temperley
\cite{Tem}. \medskip

In mathematical physics, existence of an asymptotic expansion for determinants of discrete Laplacians 
may provide a way to construct quantum field theory of a free scalar Bose field 
as a scaling limit of a Gaussian quantum field theory on a discrete simplicial complex associated to a triangulation 
of the manifold, as the mesh of the triangulation goes to zero. In fact, this intuition 
also lies behind Hawking \cite{Haw}.\medskip

In fact, asymptotics for determinants of discrete Laplacians has been studied in several instances with partial results. 
In the setup of rectilinear polygonal domains, Kenyon \cite{Ken} derived a partial asymptotic expansion 
for the determinant of the corresponding discrete Laplacian. However
existence of a constant term in that partial expansion, let alone its identification with the zeta-regularized determinant, 
remained an open problem. Other related results include Burton-Pemantle 
\cite{BuPe} and Sridhar \cite{Sri}. \medskip

In the setting of tori, the spectrum of the discrete and Hodge Laplacians is understood explicitly, 
which naturally allows for finer asymptotic results. This setting has been studied 
by Chinta, Jorgenson and Karlsson \cite{Chi}, who equated the constant term 
in the asymptotics for determinant of discrete Laplacians with the 
logarithm of the zeta-determinant, cf. also the preceeding results by Kasteleyn \cite{Kas} and Duplantier-David \cite{Dup}. 
Their analysis is based on a discussion of the discrete heat operator in terms of Bessel functions,
and is strongly rooted in the explicit structure of the spectrum. In the setting of two-dimensional tori, 
Chamaud \cite{Cha} has relaxed geometric assumptions by variational methods.  \medskip

Closely related to the question, if zeta-regularized determinant of a Hodge Laplacian may be recovered
from asymptotics of its discrete counterpart, is a problem of relating the zeta-regularized determinant
to finite eigenvalue products. Asymptotic behavior of eigenvalue products has been studied by Szeg\"o
\cite{Sze} for certain Topelitz matrices, and in fact Friedlander-Guillemin \cite{FG} have compared the
Szeg\"o and zeta-regularized determinants for zero-th order pseudo-differential operators. 
In a related work Friedlander \cite{Fri} obtained the zeta-determinant of a higher order elliptic 
pseudo-differential operator $A$ by considering the asymptotics of a determinant for some 
determinant class operator associated
to $A^{-1}$. \medskip

The presented references share the common idea of replacing the meromorphic
continuation technique used in the definition of the zeta-determinant, by analysis of asymptotic expansions 
of classical determinants. Our paper studies this question in the explicit setting of $m$-dimensional 
tori and reproves the result Chinta, Jorgenson and Karlsson proposing
an alternative independent ansatz. Our method is based on a polyhomogeneous
asymptotic expansion of the combinatorial resolvent trace jointly in ther resolvent and the discretization 
parameters. In particular this approach may be be viewed as part of a program initiated jointly with 
Lesch \cite{LesVer}, cf. also Vertman \cite{Ver} and Sauer \cite{Sa}. The main technical tool is a
careful analysis of the terms in the Euler Maclaurin formula, as well as a result on interchangeability 
of regularized limits and integrals. We focus on identification of the regularized limit of discrete 
determinants in terms of the zeta-regularized determinant, rather than studying other terms in the 
asymptotics as in \cite{Chi}. \medskip

\begin{remark} Our discussion is applicable  
beyond the setting of discrete tori. A brief look into our argument
makes apparent that our results depend on 
existence of a polyhomogeneous expansion of the combinatorial 
resolvent trace, which needs not be an exclusive feature of $m$-dimensional tori. 
We therefore expect our results to have 
applications to explicitly computable quotients of $\R^m$ 
under action of $\Z^m$ lattice subgroups. 
\end{remark}

\begin{remark}
We point out that the same principle may be applied to identify
the product of the first $N\in \N$ non-zero eigenvalues of the Laplacian on a torus (or any other 
smooth manifold with an explicitly known spectrum) with its zeta-regularized 
determinant in the regularized limit as $N\to \infty$.
\end{remark}

\subsection{Hadamard partie finie regularization}

Consider $f\in C^\infty(\R_+,\C), \R_+:=(0,\infty)$ such that for 
$x\to \infty$ in the Landau notation 
\begin{align}\label{expansion}
f(x) = \sum_{j=1}^{N-1} \sum_{k=0}^{M_j} a_{jk} x^{\A_j}\log^k(x) + 
\sum_{k=0}^{M_0} a_{0k} \log^k(x) + o(x^{\A_N}\log^{M_N}(x))
\end{align}
for some $N\in \N$ and $(\A_j)\subset \C$, such that $(\Re(\A_j))$
is a monotonously decreasing sequence with $\Re(\A_N)<0$.
Then we define the \emph{regularized limit} of $f(x)$ as $x\to \infty$ by
\begin{align*}
\LIM_{x\to \infty} f(x) := a_{00}.
\end{align*}
If $\Re(\A_N)<-1$, the integral of $f$ over $[1,R]$ admits 
an asymptotic expansion of the form \eqref{expansion} as $R\to \infty$, and we set
\begin{align*}
\regint_1^\infty f(x) dx :=\LIM_{R\to \infty} \int_1^R f(x) dx.
\end{align*}
Similar definition holds for the regularized limit at $x=0$ 
assuming an appropriate asymptotic expansion of $f(x)$ as $x\to 0$,
of the form \eqref{expansion}, where $(\Re(\A_j))$ monotonously
increasing with $\Re(\A_N)>0$. \medskip

One crucial analytic property of regularized limits and 
integrals is the following interchangeability result,
which has been presented in \cite[Lemma 3.3]{LesVer} 
under mildly stronger assumptions on the asymptotics\footnote{In contrast
to \cite[Lemma 3.3]{LesVer} we do not assume that $f(z,1)$ and $f(1,n)$ 
are smooth at $z,n=0$, and require partial asymptotics as $z,n\to \infty$ instead.
In fact the latter result is applied to some homogeneous $f(z,n)$, where 
smoothness of $f(z,1)$ and $f(1,n)$ at $z,n=0$ indeed fails.}. We provide
a proof for general asymptotic expansions of the form \eqref{expansion}.

\begin{prop}\label{interchange}
Let $f \in C^\infty(\R_+^2, \C), \R_+=(0,\infty),$ be homogeneous of order $d\in \C$ 
jointly in both variables and asymptotic expansions of the form \eqref{expansion}
in each of the variables $(z,n)\to \infty$ individually (with the other variable fixed)
\begin{equation*}
\begin{split}
&f(z,1) = \sum_{j=1}^{N-1} \sum_{k=0}^{M_j} a_{jk} z^{\A_j}\log^k(z) + 
\sum_{k=0}^{M_0} a_{0k} \log^k(z) + o(z^{\A_N}\log^{M_N}(z)), \\
&f(1,n) = \sum_{j=1}^{N'-1} \sum_{k=0}^{M'_j} b_{jk} n^{\beta_j}\log^k(n) + 
\sum_{k=0}^{M'_0} b_{0k} \log^k(n) + o(n^{\beta_{N'}}\log^{M'_N}(n)),
\end{split}
\end{equation*}
for some $N,N'\in \N$ and $(\A_j), (\beta_j) \subset \C$, such that $(\Re(\A_j)), (\Re(\beta_j))$
are monotonously decreasing sequences with
$\Re(\A_N) <-1, \Re(\beta_{N'}) <\min\{0,d+1\}$. Then
\begin{align}\label{int-lim}
\LIM_{n\to \infty} \regint_1^\infty f(z,n) dz = \regint_1^\infty \LIM_{n\to \infty} f(z,n) dz + \textup{Corr},
\end{align}
where $\textup{Corr}=\regint_0^\infty f(z,1) dz$ if $d=-1$ and zero otherwise.
\end{prop}

It is an inherent part of the statement, that the regularized limits 
and integrals in the equality \eqref{int-lim} exist.

\subsection{Polyhomogeneous expansion of combinatorial resolvent traces}

For any integer $n\in \N$ we consider the quotient space $\Z / n\Z$ 
which we refer to as a \emph{discrete circle}.
A finite product of $m$ copies of $\Z / n\Z$ defines a 
\emph{discrete torus} $\T^m_n$. which may be viewed as
a discretization of the $m$-dimensional torus manifold $\T^m$, 
given by a product of $m$ copies of $\mathbb{S}^1$.
Here, $\mathbb{S}^1 \subset \R^2$ is a circle with radius $1$. \medskip

The combinatorial Laplacian $\Delta_n$ on the discrete torus $\T^m_n$ 
is the sum of the Laplacians $\mathscr{L}_n$
on each discrete circle $\Z / n\Z$ component, 
defined for any $f:\Z / n\Z \to \R$ by the following 
difference operator
\begin{align*}
(\mathscr{L}_n f)([k]) = \frac{n^2}{4\pi^2} \left( (f([k])-f([k-1])) + (f([k])-f([k+1]))\right).
\end{align*} 
Then the spectra of $\mathscr{L}_n$ and $\Delta_n$ amount to 
\begin{equation}\label{spectra}
\begin{split}
&\sigma (\mathscr{L}_n) = \left\{\left.\frac{n^2}{\pi^2} 
\sin^2\left(\frac{\pi k}{n}\right) \right| 
k\in \N_0, k<n\right\}, \\
&\sigma (\Delta_n) = \left\{\left.\frac{n^2}{\pi^2} \sum_{i=1}^m
\sin^2\left(\frac{\pi k_i}{n}\right)\right|
k_i\in \N_0, k_i < n \ \textup{for} \ i=1,...,m\right\}.
\end{split}
\end{equation}
In this multiset notation, eigenvalues $\lambda$ appear multiple 
times according to their multiplicity $m(\lambda)$.
By the particular choice of the rescaling factor $\frac{n^2}{\pi^2}$, 
eigenvalues of $\mathscr{L}_n$ and $\Delta_n$ approximate the 
eigenvalues of the Laplace Beltrami operators on $\mathbb{S}^1$ 
and $\T^m$ respectively, as $n\to \infty$. \medskip

Consider the combinatorial Laplacian $\Delta_n$ on the discrete 
torus $\T^m_n$ and traces of the corresponding resolvent powers $\textup{Tr}(\Delta_n+z^2)^{-\A}$
for any $\A\in \N$ and $z\in \R_+$. We refer to the quantities $\textup{Tr}(\Delta_n+z^2)^{-m}$ 
and $\textup{Tr}(\Delta+z^2)^{-m}$
as \emph{resolvent traces}. The resolvent trace of the 
Laplace Beltrami operator $\Delta$ on the smooth torus $\T^m$ can be 
obtained as the limit of the combinatorial trace. In fact we have the following result.

\begin{prop}\label{trace-limit} The resolvent traces 
are smooth in $z\in \R_+$ and for $k\in \N$
$$\lim_{n\to \infty} \partial_z^{(k)} \, \textup{Tr}
(\Delta_n+z^2)^{-m} = \partial_z^{(k)} \, 
\textup{Tr}(\Delta+z^2)^{-m}.$$
\end{prop}

Our first central result is a polyhomogeneous expansion 
of the resolvent trace for the combinatorial Laplacian 
$\Delta_n$, jointly in the resolvent parameter $z\in \R_+$ 
and the discretization parameter $n\in \N$. 

\begin{theorem} \label{trace-phg}
The resolvent trace for $\Delta_n$ admits a partial polyhomogeneous expansion 
$$\textup{Tr}(\Delta_n+z^2)^{-m} = \sum_{j=0}^{m} h_{-m-j}(z,n) + H(z,n),$$
where each $h_{-m-j}\in C^\infty(\R_+^2)$ is homogeneous of order $(-m-j)$ jointly in $(z,n)$, 
$h_{-m-j}(z, 1)$ and $h_{-m-j}(1, n)$ admit an asymptotic expansion of the form
\eqref{expansion} as $z, n\to \infty$, respectively.
The remainder term satisfies $H_N(z,n) = O(z^{-2m-2})$,  as $z\to \infty$, 
uniformly in $n>0$.
\end{theorem}

\subsection{Zeta-functions and zeta-regularized determinants}

In the next step we introduce (zeta-regularized) determinants of $\Delta_n$ and $\Delta$.
In the discrete case, the determinant of $\Delta_n$ is defined here as a product of its non-zero 
eigenvalues, counted with their multiplicities, and in fact satisfies the following integral representation
\begin{equation}\label{log-det}
\log \det \Delta_n = -2 \regint_0^\infty z  \textup{Tr}(\Delta_n+z^2)^{-1} dz,
\end{equation} 
which is an immediate consequence of the following computation
\begin{equation*}
-2\regint_0^\infty \frac{zdz}{(\lambda +z^2)} = - \LIM_{R\to \infty} \LIM_{\epsilon \to 0} \, 
\Big[\log(\lambda +z^2)\Big]^{z=R}_{z=\epsilon} =  \left\{ \begin{split} &\log(\lambda), \
\textup{if} \ \lambda \neq 0, \\ &0, \ \textup{if} \ \lambda =0.\end{split}\right.
\end{equation*}
Integrating \eqref{log-det} by parts $(m-1)$ times (more precisely we perform integration by parts for the 
integral on $[\varepsilon, R]$ and take the regularized limit of the total expression as 
$\varepsilon \to 0$ and $R\to \infty$) yields
\begin{equation}\label{det-comb}
\log \det \Delta_n = -2 \regint_0^\infty  
z^{2m-1} \textup{Tr}(\Delta_n+z^2)^{-m} dz,
\end{equation}
with the boundary terms vanishing in the regularized limit. \medskip

The zeta-regularized determinant of the Laplace Beltrami 
operator $\Delta$ is obtained by the following procedure. The 
\emph{zeta-function} of $\Delta$ is defined for $\Re(s) > m/2$ by
$$
\zeta(s,\Delta):= \sum_{\lambda \in \textup{Spec}\Delta 
\backslash \{0\}} m(\lambda) \lambda^{-s},
$$
where $m(\lambda)$ denotes the multiplicity of the eigenvalue 
$\lambda$. The integral expression (cf. \cite[Section 1.3]{LesVer})
amounts after iterative integration by parts to
\begin{align*}
\zeta(s,\Delta)= 2 \, \frac{\sin \pi s}{\pi} \, \frac{\Gamma(1-s)\Gamma(m)}{\Gamma(m-s)} 
\regint_0^\infty z^{2m-2s-1} \textup{Tr}(\Delta+z^2)^{-m} dz,
\end{align*}
with the standard asymptotic expansion of the resolvent 
trace $\textup{Tr}(\Delta+z^2)^{-m}$ yields a meromorphic extension 
of $\zeta(s,\Delta)$ to the whole complex plane $\C$ with $s=0$ 
being a regular point. We define the \emph{zeta-regularized determinant} by
\begin{equation}\label{det}
\log \det\nolimits_\zeta \Delta := -\left. \frac{d}{ds} \right|_{s=0} \zeta(s,\Delta) = 
-2 \regint_0^\infty z^{2m-1}\textup{Tr}(\Delta+z^2)^{-m} dz.
\end{equation}

\subsection{Main result: Approximation by combinatorial determinants}\ 
\\[3mm] Our main result now reads as follows.

\begin{theorem}\label{main} The logarithmic determinant 
$\log \det \Delta_n$ admits a regularized limit as $n\to \infty$, 
which equals the logarithm of the zeta-determinant of the 
Laplace Beltrami operator, i.e. 
\begin{align*}
\log \det\nolimits_\zeta \Delta = \LIM_{n\to \infty} \log \det \Delta_n.
\end{align*}
\end{theorem}

A remark on the relation to the discussion in \cite{LesVer} is in order. 
By Proposition \ref{trace-limit} we may write 
\begin{align*}
\log \det\nolimits_\zeta \Delta 
&= -2 \regint_0^\infty  z^{2m-1}
\textup{Tr}(\Delta_n+z^2)^{-m} dz \\
&= -2 \regint_0^\infty  z^{2m-1}
\lim_{n\to \infty} \textup{Tr}(\Delta_n+z^2)^{-m} dz\\
&= -2 \regint_0^\infty  z^{2m-1}
\lim_{n\to \infty} \sum_{k_1=0}^{n-1} \cdots \sum_{k_m=0}^{n-1} 
\left(\w(n,k_1,...,k_m) + z^2\right)^{-m} dz,
\end{align*}
where we have introduced
\begin{align}\label{w-notation}
\w(n,k_1,...,k_m)
:=\frac{n^2}{\pi^2} \sum_{i=1}^m \sin^2\left(\frac{\pi k_i}{n}\right),
\end{align}
and hence our main Theorem \ref{main} looks as an 
application of the result on interchangeability 
of regularized sums and integrals in \cite{LesVer}. 
However in contrast to the setting considered in \cite{LesVer}, 
the individual summands $(\w(n,k_1,\ldots, k_m)+z^2)^{-m}$ 
here depend on the limiting parameter $n\in \N$ and do 
not individually admit a polyhomogeneous expansion jointly 
in the summation parameters $(k_1,...,k_m)$ and the resolvent parameter $z$.
Only the full sum in the expression of 
$\textup{Tr}(\Delta_n+z^2)^{-m}$ is polyhomogeneous,
while in the setup of \cite{LesVer} polyhomogeneity 
is rather lost after summation.

\subsection{Comparison with a theorem by Chinta, Jorgenson and Karlsson}

Chinta, Jorgenson and Karlsson \cite{Chi} use a different method 
to establish Theorem \ref{main}. In fact they consider a slightly 
more general setting where they allow the individual cyclic factors
$\Z / n\Z$ to converge to $\mathbb{S}^1$ at different rates.
\medskip

If for $\underline{n}:=(n_1,..,n_m) \in \N^m$ we write $\mathbb{T}^m_{\underline{n}}
= (\Z / n_1\Z) \times \cdots \times (\Z / n_m\Z)$, then 
in our picture the resolvent trace of the corresponding combinatorial 
Laplacian $\Delta_{\underline{n}}$ admits a partial polyhomogeneous
expansion as $(\underline{n}, z) \to \infty$. We consider here the special 
case of $n_j=n$ for all $j=1,..,m$, with the argument for the general setup 
going along the same lines. Then \cite{Chi} asserts the following.

\begin{thm}[\cite{Chi}]\label{main-chi}
Consider the rescaled combinatorial Laplacian
$\Delta'_n$ with $\sigma (\Delta_n) = \frac{n^2}{4\pi^2}\sigma(\Delta'_n)$,
which approximates the Laplace Beltrami operator $\Delta'$ on a rescaled torus
given by $m$ copies of $(2\pi)^{-1} \mathbb{S}^1$. Then
the logarithmic determinant 
$\log \det \Delta'_n$ admits a regularized limit as $n\to \infty$, 
which equals the logarithm of the zeta-determinant of the $\Delta'$, i.e. 
\begin{align*}
\log \det\nolimits_\zeta \Delta' = \LIM_{n\to \infty} \log \det \Delta'_n.
\end{align*}
\end{thm}

This statement can be easily seen to correspond to Theorem \ref{main},
which is basically an issue of conventions. Since the number of non-zero eigenvalues of $\Delta_n$
is $(n^m-1)$, and by the identity $\zeta(s, \Delta') = (2\pi)^{-2s}\zeta(s, \Delta)$, we 
obtain the following relations
\begin{align*}
&\log \det \Delta'_n = \log \det \Delta_n - (n^m-1) \log \frac{n^2}{4\pi^2}.
\\ &\log \det\nolimits_\zeta \Delta' = \log \det\nolimits_\zeta \Delta + 2 \zeta(0,\Delta) \log 2\pi.
\end{align*}
Since on closed compact manifolds $\zeta(0,\Delta) = - \dim \ker \Delta$, which 
equals $(-1)$ in the setting of tori, we conclude that Theorems \ref{main} and \ref{main-chi} are equivalent.

\section{Interchanging regularized limits and integrals}\label{sec-integral}

The main result of this section is presented in \cite[Lemma 3.3]{LesVer}, 
albeit under mildly stronger assumptions on the asymptotics. 

\begin{prop}
Let $f \in C^\infty(\R_+^2, \C), \R_+=(0,\infty),$ be homogeneous of order $d\in \C$ 
jointly in both variables and asymptotic expansions of the form \eqref{expansion}
in each of the variables $z,n\to \infty$ individually (with the other variable fixed)
\begin{equation}\label{exp}
\begin{split}
&f(z,1) = \sum_{j=1}^{N-1} \sum_{k=0}^{M_j} a_{jk} z^{\A_j}\log^k(z) + 
\sum_{k=0}^{M_0} a_{0k} \log^k(z) + o(z^{\A_N}\log^{M_N}(z)), \\
&f(1,n) = \sum_{j=1}^{N'-1} \sum_{k=0}^{M'_j} b_{jk} n^{\beta_j}\log^k(n) + 
\sum_{k=0}^{M'_0} b_{0k} \log^k(n) + o(n^{\beta_{N'}}\log^{M'_N}(n)),
\end{split}
\end{equation}
for some $N,N'\in \N$ and $(\A_j), (\beta_j) \subset \C$, 
such that $(\Re(\A_j)), (\Re(\beta_j))$
are monotonously decreasing sequences with
$\Re(\A_N) <-1, \Re(\beta_{N'}) <\min\{0,d+1\}$. Then
\begin{align}\label{int-lim}
\LIM_{n\to \infty} \regint_1^\infty f(z,n) dz = \regint_1^\infty \LIM_{n\to \infty} f(z,n) dz + \textup{Corr},
\end{align}
where $\textup{Corr}=\regint_0^\infty f(z,1) dz$ if $d=-1$ and zero otherwise.
\end{prop}

\begin{proof}
Using \eqref{exp} and $\Re(\beta_{N'}) < 0$ we find for $f(z,n)=z^df(1,n/z)$
\begin{align*}
\LIM_{n\to \infty} f(z,n) = \sum_{k=0}^{M'_0} (-1)^k b_{0k} z^d \log^k(z).
\end{align*}
For any $\A\in \C$ and $k \geq 0$ we compute iteratively
\begin{equation}\label{integral}
\int z^\A \log^k (z) dz = \left\{
\begin{split}
&\sum_{j=0}^k \frac{(-1)^jk!}{(k-j)!} \frac{\log^{k-j}(z)}{(\A+1)^{j+1}} \, z^{\A+1} , &\ \textup{if} \ \A\neq -1, \\
&\frac{\log^{k+1}(z)}{(k+1)}, &\ \textup{if} \ \A = -1.
\end{split}
\right.
\end{equation}
Consequently we find
\begin{equation}\label{right}
\regint_1^\infty  \LIM_{n\to \infty} f(z,n) dz = \left\{
\begin{split}
&\sum_{k=0}^{M'_0} (-1) b_{0k} (d+1)^{-k-1}, &\ \textup{if} \ d\neq -1, \\
&0, &\ \textup{if} \ d = -1.
\end{split}
\right.
\end{equation}
For the computation of the left hand side in \eqref{int-lim}, we employ the coordinate change rule for 
regularized integrals, cf. \cite[Lemma 2.1.4]{Les:OFT}, and obtain 
\begin{align}\label{left1}
\regint_1^\infty f(z,n) dz &= n^{d+1} \regint_{1/n}^\infty f(x,1) dx - a_{j_00}n^{d+1}\log(n) \\
&= n^{d+1} \regint_0^\infty f(x,1) dx - n^{d+1} \regint_0^{1/n} f(x,1) dx - a_{j_00}n^{d+1}\log(n),
\end{align}
where $\A_{j_0}=-1$. Using \eqref{exp} and \eqref{integral} we arrive for $d\neq -1$ at the following expansion.
\begin{align*}
\regint_0^{1/n} f(x,1) dx &\sim_{n\to \infty} 
\sum_{j=1, \beta_j\neq d+1}^{N'-1} \sum_{k=0}^{M'_j} b_{jk} n^{\beta_j-(d+1)} \sum_{j=0}^k \frac{k!}{(k-j)!} \frac{\log^{k-j}(1/n)}{(d+1-\beta_j)^{j+1}} \\
& \qquad +\sum_{j=1, \beta_j= d+1}^{N'-1} \sum_{k=0}^{M'_j} b_{jk} \frac{\log^{k+1}(1/n)}{(k+1)} \\
& \qquad \quad +\sum_{k=0}^{M'_0} b_{0k} n^{-d-1}\sum_{j=0}^k \frac{k!}{(k-j)!} \frac{\log^{k-j}(1/n)}{(d+1)^{j+1}} \\
& \qquad \quad + o(n^{\beta_{N'}-(d+1)}\log^{M'_N}(n)).
\end{align*}
If $d=-1$ the asymptotic expansion changes slightly to 
\begin{align*}
\regint_0^{1/n} f(x,1) dx &\sim_{n\to \infty} 
\sum_{j=1, \beta_j\neq d+1}^{N'-1} \sum_{k=0}^{M'_j} b_{jk} n^{\beta_j-(d+1)} \sum_{j=0}^k \frac{k!}{(k-j)!} \frac{\log^{k-j}(1/n)}{(d+1-\beta_j)^{j+1}} \\
& \qquad +\sum_{j=1, \beta_j= d+1}^{N'-1} \sum_{k=0}^{M'_j} b_{jk} \frac{\log^{k+1}(1/n)}{(k+1)} \\
& \qquad \quad +\sum_{k=0}^{M'_0} (-1)^k b_{0k} \frac{\log^{k+1}(1/n)}{(k+1)} \\
& \qquad \quad + o(n^{\beta_{N'}-(d+1)}\log^{M'_N}(n)).
\end{align*}
We may now take regularized limit of \eqref{left1} as $n\to \infty$ and find using $\Re(\beta_{N'}) < d+1$
\begin{equation}\label{left}
\LIM_{n\to \infty} \regint_1^\infty f(z,n) dz = \left\{
\begin{split}
&\sum_{k=0}^{M'_0} (-1) b_{0k} (d+1)^{-k-1}, &\ \textup{if} \ d\neq -1, \\
&\regint_{0}^\infty f(x,1) dx, &\ \textup{if} \ d = -1.
\end{split}
\right.
\end{equation}
Statement follows from comparison of \eqref{left} and \eqref{right}.
\end{proof}

\section{Polyhomogeneous expansion of the combinatorial resolvent trace}
This section is devoted to a proof of Theorem \ref{trace-phg}, using the
Euler Maclaurin summation formula in $m$ parameters.
Write for $n\in \N_0$ and $x=(x_1,...,x_m) \in \R^m$
\begin{align*}
\w(n,x) :=\frac{n^2}{\pi^2} \sum_{i=1}^m
\sin^2\left(\frac{\pi x_i}{n}\right).
\end{align*}
Consider any tuple $J=(j_1,...,j_k)\subseteq \{1,...,m\}$ of pairwise 
distinct integers with $k\leq m$. We write $|J|:=k$ and put
$$
\{x_J=0\}:= \{x\in \N^m_0 \cap 
[0,n]^m \mid x_{j_1}=...=x_{j_k}=0\}.
$$
With respect to this notation we obtain for $\A\in \N$
and $z\in \R^+$ the following representation of the resolvent 
trace
\begin{align}\label{resolvent-trace-formula}
\textup{Tr}(\Delta_n+z^2)^{-m} = \sum_{k=0}^m (-1)^k 
\sum_{|J|=k} \sum_{\{x_J=0\}} (\w(n,x) + z^2)^{-m}.
\end{align}
The summand corresponding to $k=0$ is given explicitly by
$$
S(z,n) = \sum_{x_1=0}^{n} \cdots \sum_{x_m=0}^{n} 
\left(\w(n,x_1,...,x_m) + z^2\right)^{-m}.
$$
We employ the Euler Maclaurin formula for summation in $m$ 
parameters\footnote{obtained by iterating the standard Euler Maclaurin formula in one single summation 
parameter.} to derive a polyhomogeneous expansion of $S(z,n)$, 
the other terms in the resolvent trace formula \eqref{resolvent-trace-formula} 
are treated ad verbatim. For any $M\in \N$ we obtain
\begin{equation}\label{S}
S(z,n) = \sum_{\beta\in \{1,2,3,4\}^m} P_{\beta_1,1} 
\circ \cdots \circ P_{\beta_m,m}  \left(\w(n,x_1,...,x_m) + z^2\right)^{-m},
\end{equation}
where each $P_{\beta_j,j}$ acts in the $x_j$-variable on 
$u\in C^\infty[0,\infty)$ by
\begin{equation*}
P_{\beta_j,j} u := \left\{
\begin{split}
&\int_0^n u(x_j) dx_j, &\ \textup{if} \ \beta_j=1, \\
&\sum_{k=1}^M \frac{B_{2k}}{(2k)!} \left(\partial^{(2k-1)}_{x_j}|_{x_j=n} - \partial^{(2k-1)}_{x_j}|_{x_j=0}\right)u, &\ \textup{if} \ \beta_j=2, \\
&\frac{1}{(2M+1)!} \int_0^n B_{2M+1} (x_j - [x_j]) \partial^{(2M+1)}_{x_j} u(x_j) dx_j, &\ \textup{if} \ \beta_j=3, \\
&\frac{1}{2}\left(u(x_j=n) + u(x_j=0)\right), &\ \textup{if} \ \beta_j=4.
\end{split} \right.
\end{equation*}
Here, $B_i(x)$ denotes the $i$-th Bernoulli polynomial and $B_i$ the $i$-th Bernoulli number. 
We study the $(z,n)$-behaviour of the various summands. Note that for finite $n\in \N$ all 
individual operators in the composition $P_{\beta_1,1} \circ \cdots \circ P_{\beta_m,m}$ commute.

\begin{lemma}\label{H} Put $f(n,x,z):= \w(n,x)+z^2$. For each $k\in \N$ we may write
\begin{align*}
\partial_{x_j}^{(k)}f^{-\A} = \sum\limits_{\ell =0}^{k-1} H_{k,\ell}(\partial_{x_j}f) f^{-\A-\ell-1},
\end{align*}
where $H_{1,0}(\partial_{x_j}f) = -\A \partial_{x_j}f$ and higher functionals $H_{k,\ell}$ are defined recursively by
\begin{align*}
H_{k,\ell}(\partial_{x_j}f) = \partial_{x_j} H_{k-1,\ell}(\partial_{x_j}f) + H_{k-1,\ell-1} (\partial_{x_j}f) \cdot \partial_{x_j}f,
\end{align*}
where we set $H_{k,k}, H_{k,-1}=0$. For $k$ odd we have the estimate
\begin{equation}\label{sine-factor}
\left|H_{k,\ell}(\partial_{x_j}f)\right| \leq C_{k\ell} 
\left\{\begin{split}
&n^{-k+2(\ell +1)}, \ \textup{if} \ k \geq 2 (\ell +1), \\
&x^{2(\ell +1)-k}, \ \textup{if} \ k \leq 2 (\ell +1).
\end{split}\right.
\end{equation}
\end{lemma}

\begin{proof}
The recursive structure of $\partial_{x_j}^{(k)}f^{-\A}$ follows by induction and the only intricate 
statement is the estimate of $H_{k,\ell}(g), g=\partial_{x_j}f$. Here we introduce a notion of \emph{homogeneity order}
for an expression $H_{k,\ell}(g)$ by counting for every individual summand in $H_{k,\ell}(g)$ 
each additional $\partial_{x_j}$ differentiation as lowering homogeneity
order by $(-1)$, and each factor of $g$ as increasing homogeneity order by $(+1)$. 
\medskip

With this system, 
$H_{1,0}(g) = -\A g$ is of homogeneity order $1$, and by induction 
$H_{k,\ell}(g)$ is of homogeneity order $(2(\ell+1)-k)$, where in total we count $(\ell+1)$ factors of 
$g$, and $(k-(\ell+1))$ derivatives. Note the series expansion
\begin{align*}
g = \frac{n}{\pi} \sin \left(\frac{\pi x_j}{n}\right) \cos \left(\frac{\pi x_j}{n}\right)
= n \sum_{i=0}^\infty a_i \left(\frac{\pi x_j}{n}\right)^{2i+1},\quad
\partial_{x_j} g = \sum_{i=0}^\infty a_i (2i+1) \left(\frac{\pi x_j}{n}\right)^{2i}.
\end{align*}
Consequently, each $g$ in $H_{k,\ell}(g)$ carries a factor of $x_j$, a single additional 
$\partial_{x_j}$ differentiation applied to $g$, annihilates that $x_j$ factor, 
and each further derivative adds a factor of $n^{-1}$. Consequently, 
homogeneity order simply counts the powers of $x_j$ and $n$ in 
the expression for $H_{k,\ell}(g)$, leading to the statement. 
\end{proof}

First, we derive a polyhomogeneous expansion for the summand $S(z,n)$.

\begin{prop}\label{S-expansion}
The function $S(n,z)$ admits a partial polyhomogeneous expansion 
$$S(n,z) = \sum_{j=0}^{m} h'_{-m-j}(z,n) + H'(z,n),$$
where each $h'_{-m-j}\in C^\infty(\R_+^2)$ is homogeneous of order $(-m-j)$ jointly in $(z,n)$, 
$h'_{-m-j}(z, 1)$ and $h'_{-m-j}(1, n)$ admit an asymptotic expansion of the form
\eqref{expansion} as $z, n\to \infty$, respectively. 
The remainder term satisfies $H'_N(z,n) = O(z^{-2m-2})$,  as $z\to \infty$, 
uniformly in $n>0$. Moreover, $h'_{-2m}(z,n) = z^{-2m}$.
\end{prop}

\begin{proof}
Since $\partial_{x_j}$ differentiation of odd order applied to $( \w(n,x)+z^2)^{-m}$,
always leads to factors $\sin(\pi x /n)$, summands in 
\eqref{S} with $\beta_j=2$ for some $j=1,...,m$ vanish.
We are left to consider summands in \eqref{S} with $\beta_j\in \{1,3,4\}$.
Consider first summands in \eqref{S} with $\beta_j=3$ for some $j=1,...,m$.
We will use the following basic estimates
\begin{align}
\frac{x_j}{(\w(n,x) + z^2)} \leq C \frac{x_j}{x_j^2 + z^2} &\leq C (2z)^{-1}, \label{basic1}\\
\int_{[0,n]^m} \left(\w(n,x) + z^2\right)^{-m} d^mx
 &\leq \left( \int_0^n \left( \frac{n^2}{\pi^2} \sin^2 \left(\frac{\pi x_j}{n}\right) + 
z^2\right)^{-1} dx_j\right)^m \label{basic2} \\ &= \left( \frac{2n}{\sqrt{z^2+n^2/\pi^2}}\right)^m
\leq (2\pi)^m, \nonumber
\end{align}
where we denote all universal constants in the estimates by the same symbol $C>0$.
Let $k=2M+1\geq 3m+2$ be some sufficiently large integer. 
By Lemma \ref{H} (assume $\A\geq m$) we obtain in case $k\leq 2(\ell + 1)$
\begin{equation*}
\begin{split}
&\int_{[0,n]^m} \left| H_{k,\ell}(\partial_{x_j}f) f^{-\A-\ell-1} \right|
\\ &\qquad \qquad \leq \int_{[0,n]^m} \frac{C x^{2(\ell +1)-k}}{(\w(n,x) + z^2)^{2(\ell +1)-k}}
\left(\w(n,x) + z^2\right)^{-(\A+k-(\ell +1))} d^m x
\\ &\qquad \qquad \leq C z^{-2(\ell +1)+k}\int_{[0,n]^m} 
\left(\w(n,x) + z^2\right)^{-(\A+k-(\ell +1))} d^m x
\\ &\qquad \qquad \leq C z^{-k}\int_{[0,n]^m} 
\left(\w(n,x) + z^2\right)^{-\A} d^m x
\\ &\qquad \qquad \leq C z^{-k-2(\A-m)}\int_{[0,n]^m} 
\left(\w(n,x) + z^2\right)^{-m} d^m x \leq C (2\pi)^m z^{-2\A-2},
\end{split}
\end{equation*}
where we used \eqref{basic1} in the second inequality, 
and employed \eqref{basic2} together with $k\geq 2m+2$ in the last inequality.
In the case $k> 2(\ell + 1)$, we compute similarly 
\begin{equation*}
\begin{split}
\int_{[0,n]^m} \left| H_{k,\ell}(\partial_{x_j}f) f^{-\A-\ell-1} \right|
&\leq C n^{-k+2(\ell +1)} \int_{[0,n]^m} 
\left(\w(n,x) + z^2\right)^{-\A-(\ell +1)} d^m x
\\ &\leq C \left\{\begin{split} &z^{-2\A-2}\int_{[0,n]^m} 
\left(\w(n,x) + z^2\right)^{-m} d^m x, \ \textup{if} \ \ell \geq m,
\\ &n^{-k+2(\ell +1)} z^{-2\A-2} \int_{[0,n]^m} 1 d^m x,\ \textup{if} \ \ell < m, 
\end{split}\right. \\
&\leq C' z^{-2\A-2},
\end{split}
\end{equation*}
where we used $k\geq 3m$ and \eqref{basic2} in the final estimate.
Consequently, if $\beta_j=3$ for some single $j=1,...,m$, we may estimate
for $M\geq (3m+1)/2$ uniformly in $n>0$
$$
\left| P_{\beta_1,1} 
\circ \cdots \circ P_{\beta_m,m}  \left(\w(n,x_1,...,x_m) + z^2\right)^{-\A}\right|
\leq C z^{-2\A-2}.
$$
If $\# \{\beta_j=3\}>1$, the estimates proceed along the same lines. 
Consider next the case with $\beta_j=1$ for all $j=1,...,m$. 
The corresponding summand is given by the following integral, where we 
substitute $y_j = x_j /n$
$$
P_{\beta_1,1} 
\circ \cdots \circ P_{1,m}  \left(\w(n,x_1,...,x_m) + z^2\right)^{-\A}
= n^{-2\A+m} \int_{[0,1]^m} \left(\w(1,y) + (z/n)^2\right)^{-m} d^my.
$$
the expression is homogeneous of order $(-2\A+m)$ jointly in $(n,z)$.
Existence of an asymptotic expansion for the homogeneous terms
$h'_{-2\A+m}(\cdot, 1), h'_{-2\A+m}(1, \cdot )$ follows e.g. from 
a careful application of Melrose's push forward theorem\footnote{Note that 
$\left(\w(1,y) + (z/n)^2\right)^{-m}$ lifts to a polyhomogeneous function on the
blowup space $[[0,1]^m\times \R^+, \mathscr{A}]$, blown up at the corners
$\mathscr{A}=\{(y,t) \mid y_j \in \{0,1\}, t \in \{0,\infty\}\}$. Pushforward theorem 
of Melrose \cite{Mel:COC, Mel:TAP} then yields an asymptotic expansion of the integral as $t\to 0$
or $t\to \infty$. The explicit structure of the expansion is irrelevant in our discussion.}.\medskip

It remains to discuss terms with $\beta_j=4$ for some $j \in \{1,...,m\}$.
Each $P_{4,j}$ is an evaluation operator and in fact for any family
of pairwise distinct integers $J=(j_1,...,j_k)\subset \{1,...,m\}$ we have
$$
P_{4,j_1} 
\circ \cdots \circ P_{4,j_k}  \left(\w(n,x_1,...,x_m) + z^2\right)^{-\A}
= (\w_J(n,y) + z^2)^{-\A},
$$ 
where we have introduced $y=(y_1,...,y_{m-k}):=(x_{j_{k+1}},...,x_{j_m})$ and write
\begin{align*}
\w_J(n,y) :=\frac{n^2}{\pi^2} \sum_{i=1}^{m-k}
\sin^2\left(\frac{\pi y_i}{n}\right).
\end{align*}
Consequently, analysis of the terms with $\beta_j=4$ for $j \in J$
proceeds along the lines above, with $m$ simply replaced by $(m-k)$.
This leads to homogeneous terms of homogeneity order $(-2\A+m-k)$.
Setting $\A=m$ proves the statement, once we observe that the homogeneous
term of order $(-2m)$ is given explicitly by
$$
P_{4,1} 
\circ \cdots \circ P_{4,m}  \left(\w(n,x_1,...,x_m) + z^2\right)^{-m}
= z^{-2m}.
$$ 
\end{proof}

We may now prove our first main result.

\begin{theorem}\label{mainthm} 
The resolvent trace admits a partial polyhomogeneous expansion
$$\textup{Tr}(\Delta_n+z^2)^{-m} = \sum_{j=0}^{m} h_{-m-j}(z,n) + H(z,n),$$
where each $h_{-m-j}\in C^\infty(\R_+^2)$ is homogeneous of order $(-m-j)$ jointly in $(z,n)$, 
$h_{-m-j}(z, 1)$ and $h_{-m-j}(1, n)$ admit an asymptotic expansion of the form
\eqref{expansion} as $z, n\to \infty$, respectively. 
The remainder term satisfies $H_N(z,n) = O(z^{-2m-2})$,  as $z\to \infty$, 
uniformly in $n>0$. Moreover, $h_{-2m}(z,n) = 0$.
\end{theorem}
\begin{proof}
We need to extend the computations in Proposition \ref{S-expansion} to general terms
in the expression \eqref{resolvent-trace-formula}. Consider any 
tuple $J=(j_1,...,j_k)\subset \{1,...,m\}$ of pairwise 
distinct integers with $k\leq m$ and the corresponding term 
$$
\sum_{\{x_J=0\}} (\w(n,x) + z^2)^{-m}.
$$
Put $y=(y_1,...,y_{m-k}):=(x_{j_{k+1}},...,x_{j_m})$ and note
$$
\sum_{\{x_J=0\}} (\w(n,x) + z^2)^{-m} = 
\sum_{y \in [0,n]^{m-k}} (\w_J(n,y) + z^2)^{-(m-k)-k},
$$ 
where we have introduced
\begin{align*}
\w_J(n,y) :=\frac{n^2}{\pi^2} \sum_{i=1}^{m-k}
\sin^2\left(\frac{\pi y_i}{n}\right).
\end{align*}
The computations follow for these terms along the lines 
of Proposition \ref{S-expansion} with 
$m$ replaced by $(m-k)$ and $\A=m$.
It remains to discuss the homogeneous term of order $(-2m)$.
Note by the statement of Proposition \ref{S-expansion} on the 
$(-2m)$-homogeneity terms
\begin{align*}
h_{-2m}(z,n) = \sum_{k=0}^m (-1)^k 
\sum_{|J|=k} z^{-2m} = \sum_{k=0}^m (-1)^k  \left(
\begin{array}{c} m \\ k \end{array} \right)z^{-2m} = 0.
\end{align*}
This proves the statement.
\end{proof}

\section{Convergence of the combinatorial resolvent traces}

In this section we prove Proposition \ref{trace-limit}, using 
an argument of Dodziuk \cite{Dod:FDA} on convergence of 
zeta functions.

\begin{prop} For any integer $\A\geq m$
$$\lim_{n\to \infty} \partial_z^{(k)} \, \textup{Tr}
(\Delta_n+z^2)^{-\A} = \partial_z^{(k)} \, 
\textup{Tr}(\Delta+z^2)^{-\A}.$$
\end{prop}

\begin{proof}
The argument does not depend on the explicit structure 
of the spectra for $\Delta_n$ and $\Delta$. Hence we write
\begin{align*}
&\sigma (\Delta_n) = \left\{ \lambda_{k,n} \mid k=0,...,N(n)\right\}. \\
&\sigma (\Delta) = \left\{\lambda_k \mid k \in \N_0 \right\},
\end{align*}
where both sets are ordered in an ascending order, 
$N(n)$ is the total number of $\Delta_n$-eigenvalues counted
with their multiplicities, $N(n)\to \infty$ monotonously increasing 
and $\lambda_{k,n} \to \lambda_k$ as $n\to \infty$,
for each fixed $k\in \N_0$. 
\medskip

Fix any $\varepsilon >0$. Since the resolvent trace
$$
\textup{Tr}(\Delta+z^2)^{-\A} = \sum_{k=0}^\infty (\lambda_k + z^2)^{-\A},
$$
is a convergent series, there exists $K\in \N$ sufficiently large, 
such that 
$$
\sum_{k=K}^\infty (\lambda_k + z^2)^{-\A} < \varepsilon /3.
$$
Consequently, since $(\lambda_{k,n})_{n\in \N}$ is an ascending sequence, by the 
minimax principle (cf. \cite{Dod:FDA}), we may estimate
$$
0 \leq \sum_{k=K}^\infty (\lambda_k + z^2)^{-\A} - \sum_{k=K}^{N(n)} (\lambda_{k,n} + z^2)^{-\A}
< 2 \varepsilon /3.
$$
Finally, given 
convergence of the spectrum, choose 
$n\geq n_0$, such that
$$
0 \leq \sum_{k=0}^{K} \left((\lambda_k + z^2)^{-\A} - (\lambda_{k,n} + z^2)^{-\A}
\right) < \varepsilon /3.
$$ 
Thus, for any given $\varepsilon >0$, there exists 
$n\in \N$ sufficiently large, such that
$$
0 \leq \textup{Tr}(\Delta+z^2)^{-\A} - \textup{Tr}(\Delta_n+z^2)^{-\A} = 
\sum_{k=0}^\infty (\lambda_{k} + z^2)^{-\A}
- \sum_{k=0}^{N(n)}(\lambda_{k,n} + z^2)^{-\A} < \varepsilon.
$$
Hence the combinatorial resolvent trace $\textup{Tr}(\Delta_n+z^2)^{-\A}$ converges
to $\textup{Tr}(\Delta+z^2)^{-\A}$ as $n\to \infty$. The general statement follows
from the observation
\begin{align*}
&\partial_z \textup{Tr} (\Delta_n+z^2)^{-\A} = 2z \textup{Tr} (\Delta_n+z^2)^{-\A-1}, \\
&\partial_z \textup{Tr} (\Delta+z^2)^{-\A} = 2z \textup{Tr} (\Delta+z^2)^{-\A-1}.
\end{align*}

\end{proof}

\section{Proof of the main result}

We may now prove our main theorem. 

\begin{theorem}The logarithmic determinant $\log \det \Delta_n$ admits a regularized limit as $n\to \infty$ and 
\begin{align*}
\log \det\nolimits_\zeta \Delta = \LIM_{n\to \infty} \log \det \Delta_n.
\end{align*}
\end{theorem}

\begin{proof}
By \eqref{det} and Proposition \ref{trace-limit}
\begin{align*}
\log \det\nolimits_\zeta \Delta 
&= -2 \regint_0^\infty  z^{2m-1}
\textup{Tr}(\Delta+z^2)^{-m} dz \\
& = -2 \regint_0^\infty  z^{2m-1}
\lim_{n\to \infty}\textup{Tr}(\Delta_n+z^2)^{-m} dz \\
& = -2 \regint_0^\infty  z^{2m-1}
\LIM_{n\to \infty}\textup{Tr}(\Delta_n+z^2)^{-m} dz.
\end{align*}
By Theorem \ref{trace-phg}, $\textup{Tr}(\Delta_n+z^2)^{-m}$ 
admits a polyhomogeneous expansion 
$$
\textup{Tr}(\Delta_n+z^2)^{-m} = \sum_{j=0}^{m} h_{-m-j}(z,n) + H(z,n),
$$
where each $h_{-m-j}\in C^\infty(\R_+^2)$ is homogeneous of order 
$(-m-j)$ jointly in $(z,n)$, and the remainder term satisfies
$H(z,n)=O(z^{-2m-2})$ as $z\to \infty$, uniformly in $n>0$. 
By Proposition \ref{interchange}
\begin{align*}
-2 \regint_1^\infty  z^{2m-1}
 \LIM_{n\to \infty} \sum_{j=0}^{m} h_{-m-j}(z,n) dz 
= &-2 \LIM_{n\to \infty} \regint_1^\infty  z^{2m-1}
\sum_{j=0}^{m} h_{-m-j}(z,n) dz \\ 
&-2 \regint_0^\infty z^{2m-1}h_{-2m}(z,1) dz.
\end{align*}

Each homogeneous term $h_{-m-j}$
admits a regularized limit as $n\to \infty$ and 
hence, by Proposition \ref{trace-limit}, the remainder term $H(z,n)$
admits a regularized limit as $n\to \infty$ as well. Since the 
estimate $H(z,n)=O(z^{-2m-2})$ for $z\to \infty$ is uniform in $n>0$,
the regularized limit $\LIM_{n\to \infty}H(z,n)$ is in fact a true limit 
$\lim_{n\to \infty} H(z,n)$. We find by dominated convergence
\begin{align*}
-2 \int_1^\infty  z^{2m-1}
\lim_{n\to \infty} H(z,n) dz 
= -2 \lim_{n\to \infty} \int_1^\infty  z^{2m-1}
H(z,n) dz.
\end{align*}
Consequenty, we arrive at the following intermediate result
\begin{equation}\label{change1}
\begin{split}
-2 \regint_1^\infty  z^{2m-1}
\textup{Tr}(\Delta+z^2)^{-m} dz 
 = &-2 \regint_1^\infty  z^{2m-1}
\lim_{n\to \infty}\textup{Tr}(\Delta_n+z^2)^{-m} dz \\
 = &-2 \LIM_{n\to \infty} \regint_0^\infty  z^{2m-1}
\textup{Tr}(\Delta_n+z^2)^{-m} dz\\
&-2 \regint_0^\infty z^{2m-1}h_{-2m}(z,1) dz.
\end{split}
\end{equation}
Note that $\textup{Tr}(\Delta+z^2)^{-m}\sim z^{-2m}$ as $z\to 0$
and hence $z^{2m-1}\textup{Tr}(\Delta+z^2)^{-m}$ is not integrable 
at $z=0$ due to non-trivial kernel of $\Delta$. Similar statement holds
for $\textup{Tr}(\Delta_n+z^2)^{-m}$. Subtracting 
the contribution from harmonic functions\footnote{Note from \eqref{spectra} that 
$\ker \Delta_n$ and $\ker \Delta$ are both one-dimensional.}, we find
\begin{equation*}
\begin{split}
-2 \regint_0^1  z^{2m-1}
\textup{Tr}(\Delta+z^2)^{-m} dz 
= -2 \int_0^1  z^{2m-1}
\left(\textup{Tr}(\Delta+z^2)^{-m} - z^{-2m}\right) dz, 
\\ -2 \regint_0^1  z^{2m-1}
\textup{Tr}(\Delta_n+z^2)^{-m} dz 
= -2 \int_0^1  z^{2m-1}
\left(\textup{Tr}(\Delta_n+z^2)^{-m} - z^{-2m}\right) dz, 
\end{split}
\end{equation*}
where the integrals on the right hand side exist in the usual sense.
Consequently, replacing regularized integrals with the usual integrals, 
we may interchange integrals and limits. This yields
\begin{equation}\label{change2}
\begin{split}
-2 \regint_0^1  z^{2m-1}
\textup{Tr}(\Delta+z^2)^{-m} dz 
 = &-2 \regint_0^1  z^{2m-1}
\lim_{n\to \infty}\textup{Tr}(\Delta_n+z^2)^{-m} dz \\
 = &-2 \lim_{n\to \infty} \regint_0^1 z^{2m-1}
\textup{Tr}(\Delta_n+z^2)^{-m} dz.
\end{split}
\end{equation}
By \eqref{change1} and \eqref{change2}, we find in view of the formula
\eqref{det}
\begin{align}\label{almost}
\log \det\nolimits_\zeta \Delta = \LIM_{n\to \infty} \log \det \Delta_n
-2 \regint_1^\infty z^{2m-1}h_{-2m}(z,1) dz.
\end{align}
The statement now follows from the fact that 
$h_{-2m}(z,n)= 0$ by Theorem \ref{mainthm}.
\end{proof}

\section*{Acknowledgements}
The author gratefully acknowledges helpful discussions with Matthias Lesch, Nicolai Reshetikhin and Daniel Grieser. 
He also gratefully acknowledges financial support by the Hausdorff Center for Mathematics in Bonn and
by the Mathematical Institute at M\"unster University.

\providecommand{\bysame}{\leavevmode\hbox{\hrulefill}\thinspace}
\providecommand{\MR}{\relax\ifhmode\unskip\space\fi MR }
\providecommand{\MRhref}[2]{%
  \href{http://www.ams.org/mathscinet-getitem?mr=#1}{#2}
}
\providecommand{\href}[2]{#2}


\begin{thebibliography}{\textsc{DuDa88}}

\bibitem[\textsc{BuPe93}]{BuPe}
\textsc{R.~Burton} and \textsc{R.~Pemantle}, \emph{Local characteristics,
  entropy and limit theorems for spanning trees and domino tilings via
  transfer-impedances}, Ann. Probab. \textbf{21} (1993), no.~3, 1329--1371.
  \MR{1235419 (94m:60019)}

\bibitem[\textsc{Cha06}]{Cha} \textsc{L. Chaumard},
\emph{Discrétisation de zeta-déterminants d'opérateurs de Schrödinger sur le  
tore}, Bull. Soc. Math. France \textbf{134} (2006), no. 3, 327?355.

\bibitem[\textsc{Che79}]{Che} \textsc{J. Cheeger}, 
\emph{Analytic Torsion and the Heat Equation},  Ann. of Math.(2) {\bf109} (1979) no. 2, 259--322.


\bibitem[\textsc{CJK10}]{Chi}
\textsc{G.~Chinta}, \textsc{J.~Jorgenson}, and \textsc{A.~Karlsson}, \emph{Zeta
  functions, heat kernels, and spectral asymptotics on degenerating families of
  discrete tori}, Nagoya Math. J. \textbf{198} (2010), 121--172. \MR{2666579
  (2011i:58052)}

\bibitem[Dod76]{Dod:FDA}
\textsc{J. Dodziuk}, \emph{Finite-difference approach to the {H}odge theory of
  harmonic forms}, Amer. J. Math. \textbf{98} (1976), no.~1, 79--104.
  \MR{0407872 (53 \#11642)}


\bibitem[\textsc{DuDa88}]{Dup}
\textsc{B.~Duplantier} and \textsc{F.~David}, \emph{Exact partition functions
  and correlation functions of multiple {H}amiltonian walks on the {M}anhattan
  lattice}, J. Statist. Phys. \textbf{51} (1988), no.~3-4, 327--434. \MR{952941
  (89m:82005)}

\bibitem[\textsc{FrGu08}]{FG}
\textsc{L. Friedlander}, \textsc{V. Guillemin},
\emph{Determinants of zeroth order operators},
J. Diff. Geom. \textbf{78}, 1, (2008), 1-12.


\bibitem[\textsc{Fri89}]{Fri}
\textsc{L. Friedlander}, 
\emph{The asymptotics of the determinant function for a class of operators}, 
Proc. Amer. Math. Soc., \textbf{107}, 1989, 169--178


\bibitem[\textsc{Haw77}]{Haw}
\textsc{S. W. Hawking} \emph{Zeta function regularization of path integrals in curved spacetime},
Comm. Math. Phys. \textbf{55}, 2 (1977), 133-148.


\bibitem[\textsc{Kas61}]{Kas}
\textsc{P.~W. Kasteleyn}, \emph{The statistics of dimers on a lattice, i. the
  number of dimer arrangements on a quadratic lattice}, Physica \textbf{27}
  (1961), 1209--12225.

\bibitem[\textsc{Ken00}]{Ken}
\textsc{R.~Kenyon}, \emph{The asymptotic determinant of the discrete
  {L}aplacian}, Acta Math. \textbf{185} (2000), no.~2, 239--286. \MR{1819995
  (2002g:82019)}

\bibitem[\textsc{Kir47}]{Kir}
\textsc{G.~Kirchhoff}, \emph{\"uber die aufl\"osung der gleichungen, auf welche
  man bei der untersuchung der linearen verteilung galvanischer sterne
  gef\"uhrt wird}, Ann. Phys. Chem. \textbf{72} (1847), 497--508.

\bibitem[\textsc{Les97}]{Les:OFT}
\textsc{M.~Lesch}, \emph{Operators of {F}uchs type, conical singularities, and
  asymptotic methods}, Teubner-Texte zur Mathematik [Teubner Texts in
  Mathematics], vol. 136, B. G. Teubner Verlagsgesellschaft mbH, Stuttgart,
  1997. \texttt{arXiv:dg-ga/9607005v1}, \MR{1449639 (98d:58174)}

\bibitem[\textsc{LeVe13}]{LesVer}
\textsc{M.~Lesch} and \textsc{B.~Vertman}, \emph{Regularizing infinite sums of
  zeta-determinants},  \texttt{arXiv:1306.0780 [math.SP]}.

\bibitem[\textsc{Mel92}]{Mel:COC}
\textsc{R. Melrose}, \emph{Calculus of conormal distributions on manifolds with corners}, 
Intl. Math. Research Notices \textbf{3}  (1992), 51-61.

\bibitem[\textsc{Mel93}]{Mel:TAP}
\bysame,  \emph{The Atiyah-Patodi-Singer index theorem} Research Notes in Math., 
\textbf{4}, A K Peters, Massachusetts (1993)

\bibitem[\textsc{M{\"u}l78}]{Mue}
\textsc{W.~M{\"u}ller}, \emph{Analytic torsion and {$R$}-torsion of
  {R}iemannian manifolds}, Adv. in Math. \textbf{28} (1978), no.~3, 233--305.
  \MR{498252 (80j:58065b)}


\bibitem[\textsc{RaSi71}]{RS} 
\textsc{D.B. Ray} and \textsc{I.M. Singer} 
\emph{R-torsion and the Laplacian on Riemannian manifolds}, 
Adv. Math. {\bf 7} (1971), 145-210.

\bibitem[\textsc{ReVe13}]{ReVe}
\textsc{N.~Reshetikhin} and \textsc{B.~Vertman}, \emph{Combinatorial quantum
  field theory and gluing formula for determinants}.

\bibitem[\textsc{Sau13}]{Sa}
\textsc{B.~Sauer}, \emph{On the resolvent trace of multi-parametric
  {S}turm-{L}iouville operators}, Master thesis, Bonn (2013).

\bibitem[\textsc{Sri13}]{Sri}
\textsc{A.~Sridhar}, \emph{Asymptotic determinant of discrete Laplace-Beltrami operators},
preprint arXiv:1501.02057 (2015)

\bibitem[\textsc{Sze15}]{Sze}
\textsc{G. Szeg\"o}, \emph{Ein Grenzwertsatz \"uber die Toeplitzschen 
Determinanten einer reellen positiven Funktion}, Math. Ann. 76 \textbf{4}: 490?503, (1915)

\bibitem[\textsc{Tem74}]{Tem}
\textsc{H.~N.~V. Temperley}, \emph{Enumeration of graphs on a large periodic
  lattice}, Combinatorics ({P}roc. {B}ritish {C}ombinatorial {C}onf., {U}niv.
  {C}oll. {W}ales, {A}berystwyth, 1973), Cambridge Univ. Press, London, 1974,
  pp.~155--159. London Math. Soc. Lecture Note Ser., No. 13. \MR{0347616 (50
  \#119)}

\bibitem[\textsc{Ver13}]{Ver}
\textsc{B.~Vertman}, \emph{Multiparameter resolvent trace expansion for
  elliptic boundary problems},  \texttt{arXiv:1301.7293 [math.SP]}.

\end{thebibliography}
\end{document}